\documentclass{amsart}
\usepackage{marktext}
\usepackage{gimac}
\usepackage{amsmath}
\usepackage{mathtools}

\usepackage[tight]{subfigure}

\colorlet{darkred}{red!40!black}

\newdelim{\ip}{\langle}{\rangle}

\makeatletter
\newcommand{\leqnomode}{\tagsleft@true\let\veqno\@@leqno}
\newcommand{\reqnomode}{\tagsleft@false\let\veqno\@@eqno}
\newcommand{\mylabel}[2]{\def\@currentlabel{#2}\label{#1}}
\makeatother

\newcommand{\TT}{\mathbb{T}}

\makeatletter
\newcommand{\vast}{\bBigg@{4}}
\newcommand{\Vast}{\bBigg@{5}}

\makeatother

\begin{document}
\title[Dissipation enhancement]%
{Dissipation enhancement for a degenerated parabolic equation}
\author[Feng]{Yu Feng}
\address{Beijing International Center for Mathematical Research, Peking University, No. 5 Yiheyuan Road Haidian District, Beijing, P.R.China 100871}
\email{fengyu@bicmr.pku.edu.cn}

\author[Hu]{Bingyang Hu}
\address{%
  Department of Mathematics, Purdue University, 150 N. University St.,W. Lafayette, IN 47907, USA}
\email{hu776@purdue.edu}

\author[Xu]{Xiaoqian Xu}
\address{%
  Zu Chongzhi Center for Mathematics and Computational Sciences, Duke Kunshan University, China}
\email{xiaoqian.xu@dukekunshan.edu.cn}

\date{\today}

\subjclass[2010]{Primary 76R05, Secondary 35B27, 35B44, 35Q35}%

\keywords{Dissipation enhancement, mixing, non-linear dissipation time}%
\thanks{}

\begin{abstract}
In this paper, we quantitatively consider the enhanced-dissipation effect of the advection term to the parabolic $p$-Laplacian equations. More precisely, we show the mixing property of flow for the passive scalar enhances the dissipation process of the $p$-Laplacian in the sense of $L^2$ decay, that is, the $L^2$ decay can be arbitrarily fast. The main ingredient of our argument is to understand the underlying iteration structure inherited from the parabolic $p$-Laplacian equations. This extends the dissipation enhancement result of the advection diffusion equation by Yuanyuan Feng and Gautam Iyer into a non-linear setting. 
\end{abstract}

\maketitle

\section{Introduction}

In the study of incompressible fluids, one fundamental phenomenon that arises in a wide variety of applications is \emph{dissipation enhancement}, whose mechanism, in general, comes from the following two primary sources: \emph{mixing}, which induces filamentation and facilitates the formation of small scales (namely, high frequencies); and \emph{diffusion}, which efficiently damps small scales and accelerates the dissipation process.

Without diffusion, the behavior of passive scalar advected by incompressible mixing flows has been extensively studied in recent decades. In particular, some researchers developed multi-scale norms to quantify mixing (see \cites{LinThiffeaultEA11,Thiffeault12}) and analyzed the decay rate of such norms for incompressible flows (see, e.g.,\cites{IyerKiselevEA14, YaoZlatos17,ElgindiZlatos19,AlbertiCrippaEA16,LunasinLinEA12}). On the other hand, an apriori limit to the resolution via mixing is inaccessible in this regime.

\emph{In the presence of diffusion, the effects of mixing may be enhanced, balanced, or even suppressed by diffusion} (see, e.g.,  \cite{FengIyer19}). One of the most famous models used to study the interaction between these two sources is the advection-diffusion equation
\begin{equation}\label{eq:AD}
\begin{cases}
    \partial_t\vartheta +u\cdot\grad\vartheta-\nu\lap\vartheta=0 \\
    \\
    \vartheta(0,x)=\vartheta_0\in L_{0}^2,
\end{cases}
\end{equation}
where $u$ stands for a time-dependent or time-independent incompressible velocity field, $\nu>0$ presents the strength of diffusion, proportional to the inverse P\'{e}clet number and $L_0^2$ refers the collection of all $L^2(\T^d)$ functions with mean zero. 

The study of the dissipation enhancement for the linear advection equations is a popular topic in recent years. It dates back to the celebrated work \cite{ConstantinKiselevEA08} by Constantin, Kiselev, Ryzhik and Zlato\v{s} in 2008, in which, they first found an equivalent condition for the time-independent $u$ to enhance the dissipation effect, in the sense of the $L^2$ norm faster decay of the solution, on a compact manifold like torus. Later, there have been many other extension along this direction (see \cites{Zlatos10, wei2021SciChinaMath}). In particular, in \cite{ZelatiDelgadinoElgindi20,FengIyer19}, the authors studied the dissipation enhancement of \eqref{eq:AD} under various mixing conditions, and notably their methods lead to an explicit quantitative bound when the mixing flow $u$ is assumed to be polynomial and exponential. 

As some byproducts, understanding dissipation enhancement can help us study stabilization phenomena of the singularity. We refer the interested reader  \cites{FannjiangKiselevEA06,BerestyckiKiselevEA10,KiselevXu16,BedrossianHe17,He18,HeTadmor19,fengfengIyerThiffeault2020nonlinearscience,feng2020global, feng2020suppression} and the reference therein for more details. 

The goal of this paper is to study the effect of mixing on dissipation enhancement to a non-linear diffusion model. One natural nonlinear model which resembles \eqref{eq:AD} is the $p$-Laplacian dissipation. More precisely,  we consider the following system: for $p>2$, $u$ an incompressible flow on $\T^d$ and $\nu>0$, 
 \begin{equation} \label{eq:AD2}
\begin{cases}
\partial_t \theta+\paren{u(t) \cdot \grad} \theta-\nu \grad \cdot \left(\abs{\grad \theta }^{p-2} \grad \theta \right)=0; \\
\\
\theta(x,0)=\theta_{0}(x) \in L_0^2. 
\end{cases}
\end{equation} 
Here $p$ measures the level of diffusion, and it is clear that when $p=2$, \eqref{eq:AD2} reduces the regular Laplacian dissipation.

The interest in considering such equation is motivated by the study of following advective thin-film type equation: 
\begin{equation} \label{20200716eq01}
    h_t + u\cdot\grad h + A_1\lap h + A_2\lap^2 h + A_3\grad\cdot(\abs{\grad h}^{p-2}\grad h) = g, \quad p>2,
\end{equation}
where $h(t, x)$ denotes the height of a film in epitaxial growth with $g(t,x)$ being the deposition flux and $A_1, A_2, A_3 \in \R$. The spatial derivatives in the above equation have the following physical interpretations:
\begin{itemize}
    \item [(1). ] $u\cdot\grad h$ : transportation under a velocity field $u$.
    \item [(2). ] $A_1\lap h$: diffusion due to evaporation-condensation~\cites{edwards1982surface,mullins1957theory}; 
    \item [(3). ] $A_2\lap^2 h$: capillarity-driven surface diffusion~\cites{herring1999surface,mullins1957theory}; 
    \item [(4). ] $A_3\grad\cdot(\abs{\grad h}^{p-2}\grad h)$: (upward) hopping of atoms~\cite{sarma1992solid}.
\end{itemize}
The dissipation enhancement phenomenon to the linear diffusion such as $A_1\lap h$ and $A_2\lap^2 h$ have been well studied in \cites{FengIyer19,ZelatiDelgadinoElgindi20}. However, as far as we know, the study of dissipation enhancement to the nonlinear diffusion $A_3\grad\cdot(\abs{\grad h}^{p-2}\grad h)$ is still less understood. Therefore, we separate it from the thin film equation \eqref{20200716eq01} and consider equation \eqref{eq:AD2} for simplicity. It is worth mentioning that the case $A_3$ takes negative values in \eqref{20200716eq01} has also been investigated recently in \cites{ishige2020blowup,feng2020suppression}.

Further more, \eqref{eq:AD2} itself has wide applications in the mathematical modeling of various real-world processes, such as the flows of electrorheological or thermo-rheological fluids~\cites{antontsev2014asymptotic,antontsev2006stationary,ruzicka2000electrorheological}, the problem of thermistors~\cites{zhikov2008solvability}, and processing of digital images~\cite{chen2006variable}.

In this paper, we consider how the advection terms can help the $p$-Laplacian term to dissipate the solution's $L^2$ norm. The novelty of this paper is two fold.

\begin{enumerate}
    \item [1.] The main difference between \eqref{eq:AD2} and the linear model \eqref{eq:AD} is that its solution becomes degenerate or singular at the points where $\abs{\grad\theta}=0$, preventing one from expecting classical solutions. We overcome this difficulty by showing the weak solutions satisfy specific regularity properties (see Section \ref{20210223sec01}), and then we can pick up test functions to obtain desired energy identities (see, e.g., \eqref{e:energy identity1} and \eqref{e:energy identity2});
    
    \medskip
    
    \item [2.] The solution operator of \eqref{eq:AD2} is non-linear, which suggests that the semigroup approach inherited under the linear model \eqref{eq:AD} might fail in our case. Instead of using the semigroup structure, we explore a new iteration structure (see Lemma \ref{20210121lem021}) underlying the equation \eqref{eq:AD2}, which essentially plays the same role as the semigroup structure in the linear model case.
\end{enumerate}

The plan of this paper is as follows. We begin by defining mixing rates, introduce the non-linear dissipation time, and state our main results, as well as some applications in Section 2. Section 3 is devoted to explore several properties for the weak solution of the equation \eqref{eq:AD2}. As a consequence, we give a prior estimate for the non-linear dissipation time. Finally, in Section 4, we prove the quantitative dissipation time bound for the $p$-Laplacian dissipation. The proof involves a study of the underlying iteration structure of the equation \eqref{eq:AD2}.

\subsection*{Acknowledgments}
The authors would like to thank the anonymous referees for their helpful comments.

\section{Main results} 

Let $\TT^d:=[0, 1]^d$ be $d$-dimensional torus, and $u$ be a smooth, time dependent, divergence free vector field on  $\TT^d$. Let further, 
$$
0<\lambda_1 \le \lambda_2 \le \dots
$$
be the eigenvalues of $-\lap$ on $\TT^d$. Moreover, we denote $L^p:=L^p(\TT^d)$ with norm denoted as $\|\cdot\|_p$. And for any $\alpha \in \R$, recall the \emph{homogeneous Sobolev space of order $\alpha$} is given by 
$$
\dot{H}^{\alpha}=\dot{H}^{\alpha}(\TT^d):=\left\{f=\sum_i a_i e_i: \norm{f}_{\dot{H}^{\alpha}}^2:=\sum_i \lambda_i^\alpha \abs{a_i}^2<\infty\right\}, 
$$
where in the above definition, $e_i$ is the normalized eigvenvector corresponding to the eigenvalue $\lambda_i$, $i \ge 1$. Note that under this formulation, $\|\grad(\cdot)\|_2=\|\cdot\|_{\dot{H}^{1}}$. 

The goal of this paper is to introduce and study the dissipation enhancement of advection to a nonlinear diffusivity. Let $\nu>0$ be the strength of the diffusion. Now for $p>2$ and each time $s \ge 0$, we consider following non-linear parabolic equation with gradient nonlinearity on the $d$-dimensional torus with advection of an incompressible vector field $u(t)$: 
\begin{equation} \label{maineq}
\begin{cases}
\partial_t \theta_s+\paren{u(t) \cdot \grad }\theta_s-\nu \lap_p \theta_s =0; \\
\\
\theta_s(t)=\theta_{s, 0}, \quad \quad  t=s. 
\end{cases}
\end{equation} 
for $t>s$, with initial data $\theta_{s, 0}=\theta_0(s)$. Here,
\begin{enumerate}
    \item [1.] $\nabla$ is the covariant derivative; 
    
    \medskip
    
    \item [2.] $\lap_p \theta_s:= \grad \cdot \left(\abs{\grad \theta_s }^{p-2} \grad \theta_s \right)$ is the $p$-Laplacian;
    
    \medskip
    
    \item [3.] $\theta_{s, 0}=\theta_{0}(s)$, where $\theta_0(\cdot)$ is the solution of \eqref{maineq} with $s=0$ and initial data $\theta_{0, 0} \in L^2_0(\TT^d)$, which is the space of $L^2$ integrable functions on $\TT^d$ with mean zero.
\end{enumerate} 

\begin{remark}
We make a remark that the solution of \eqref{maineq} should be understood in weak sense. Moreover, all these solutions have certain regularity (see Theorem \ref{regthm}), which, in particular, guarantees that $\theta_{s, 0}$ is a measurable function and hence \eqref{maineq} is well-defined. Furthermore, we can actually see that $\theta_{s, 0}$ also belongs to $L_0^2(\TT^d)$ (see Corollary \ref{20210223cor01}).  
\end{remark}

We are interested in the behavior of solutions of \eqref{maineq} for $\nu \ll 1$ and a fixed initial data $\theta_{0, 0}$. The prototype of our model is the \emph{linear diffusion equation}  . One typical example would be 
\begin{equation} \label{20201222eq01}
\begin{cases}
\partial_t \vartheta_s+\paren{u(t) \cdot \grad} \vartheta_s-\nu \left(-\lap \right)^\alpha \vartheta_s =0; \\
\\
\vartheta_s(t)=\vartheta_{s, 0}, \quad \quad  t=s,
\end{cases}
\quad \alpha>0,
\end{equation} 
for $t>s$, with initial data $\vartheta_{s, 0} \in L_0^2(\TT^d)$, where $\lap$ is the Laplace-Beltrami operator on $\TT^d$. Observe that when $\alpha=1$, \eqref{20201222eq01} becomes the \emph{advection diffusion equation} (see, e.g., \cite{ConstantinKiselevEA08}); when $\alpha=2$, \eqref{20201222eq01} refers to the \emph{advective hyperdiffusion equation}. In the sequel, our interest will lie in the case when $\alpha=1$, as this is exactly our main equation \eqref{maineq} with $p=2$, which can be viewed as the ``endpoint case'' for the non-linear model. The dissipation enhancement in the advection diffusion equation (when $\alpha=1$) has been studied a lot in the recent years (see \cites{ConstantinKiselevEA08,FengIyer19,ZelatiDelgadinoElgindi20}), and the results therein can be easily adapted to the case when $\alpha\ge 1$. For linear fractional diffusions on the torus i.e. $0<\alpha<1$, the dissipation enhancement phenomenon has also been characterized in the Appendix B of \cite{hopf2018aggregation}.

One crucial concept to describe the behavior of the linear model \eqref{20201222eq01} when $\nu \ll 1$ is the \emph{linear dissipation time} of the flow $u$ for the linear models (i.e.  the time required for the system to dissipate a constant fraction of its initial energy) is given by 
\begin{equation} \label{20201222eq02}
\tau_d:=\sup_{s \in \R} \left\{\inf \left\{ t-s \bigg | t \ge s, \ \textrm{and} \  \|\vartheta_s(t)\|_2 \le \frac{\|\vartheta_{s, 0}\|_2}{e} \ \textrm{for all} \ \theta_{s, 0} \in L^2_0(\TT^d) \right\}\right\}.
\end{equation} 
Note that, for example (say, $\alpha=1$), since the solution of \eqref{20201222eq01} is strong, we are able to multiply $\vartheta_s$ on both sides of the equation in \eqref{20201222eq01} and integrate over $\TT^d$ to see
\begin{equation} \label{20210220eq01}
\frac{1}{2} \partial_t \|\vartheta_s(t)\|_2^2+\nu \|\vartheta_s\|_{\dot{H}^1}^2=0,
\end{equation}
and hence
\begin{equation} \label{20200221eq01}
\|\vartheta_s(t)\|_2 \le e^{-\nu\lambda_1 (t-s)} \|\vartheta_{s, 0}\|_2
\end{equation}
This implies
    \begin{equation} \label{20201222eq03}
    \tau_d \le \frac{1}{\nu \lambda_1},
    \end{equation}
where we recall that $\lambda_1$ is the principal eigvenvalue of $-\lap$ on $\TT^d$. Moreover, it turns out this is the best one can hope, that if $u$ is only assumed to be incompressible (see \cite{ConstantinKiselevEA08}). 

An important feature for $\tau_d$ is that when the flow $u$ is assumed to be \emph{mixing}, one can improve the estimate \eqref{20201222eq03} into, heuristically, 
\begin{equation} \label{20210220eq02}
\tau_d \le \frac{C}{\nu \lambda_N},
\end{equation}
where $C$ is a universal constant and $N$ is a constant which only depends on the equation and the mixing condition (see, e.g. \cites{FengIyer19, ZelatiDelgadinoElgindi20} for a more comprehensive treatment). It's worth to mention that the upper bound of $\tau_d$ obtained in \cite{FengIyer19} is not as sharp as the ones obtained from \cite{ZelatiDelgadinoElgindi20}, the reason of this will be discussed in detail in Remark~\ref{rmk:diff estimate}. Now, let us first recall the mixing condition. 

\begin{definition}\label{def:mixing estimates}
Let $h: [0, \infty) \to (0, \infty)$ be a strictly decreasing function that vanishes at infinity, and $0 \le \alpha<\infty, \beta \ge 0$. Let further, $\varphi_{s, t}: \TT^d \to \TT^d$ be the flow map of $u$ defined by
$$
\partial_t (\varphi_{s, t})=-u\left(\varphi_{s, t}\right) \quad \textrm{and} \quad \varphi_{s, s}=\textrm{Id}.
$$
\begin{enumerate}
    \item [(1)] We say that the flow $u$ is \emph{strongly $\alpha, \beta$ mixing with rate function $h$} if for all $f \in \dot{H}^{\alpha}, g \in \dot{H}^\beta$, we have
    \begin{equation} \label{strongmix}
    \left| \langle f \circ \varphi_{s, t}, g \rangle \right| \le h(t-s) \|f\|_{\dot{H}^\alpha}\|g\|_{\dot{H}^\beta}.
    \end{equation}
    Or equivalently, for all $f \in \dot{H}^\alpha$, there holds
    $$
    \|f \circ \varphi_{s, t}\|_{\dot{H}^{-\beta}} \le h(t-s) \|f\|_{\dot{H}^\alpha}; 
    $$

    \item [(2)] We say that $\varphi$ is \emph{weakly $\alpha, \beta$ mixing with rate function $h$} if for all $f \in \dot{H}^\alpha, g \in \dot{H}^\beta$, we have
    \begin{equation} \label{weakmix}
    \left( \frac{1}{t-s} \int_s^t \left| \langle f \circ \varphi_{s, r}, g \rangle \right| ^2 dr\right)^{\frac{1}{2}} \le h(t-s) \|f\|_{\dot{H}^\alpha}\|g\|_{\dot{H}^\beta}. 
    \end{equation}
\end{enumerate}
\end{definition}

\begin{remark}
\begin{enumerate}
    \item [1.]The decay estimates \eqref{strongmix} and \eqref{weakmix} arise from the literature of dynamical systems, and the traditional choice therein is to use H\"{o}lder norms. However, for convenience to our purposes, we follow the research by Fannjiang et al. in \cites{FannjiangWoowski03,FannjiangNonnenmacherEA04,FannjiangNonnenmacherEA06} and use Sobolev norms instead. 
    \medskip
   \item [2.] The analog of Definition \ref{def:mixing estimates} in discrete time was considered in \cite{FengIyer19} as well, which has been investigated more deeply than the continuous case (See Section 4 and Appendix A of \cite{FengIyer19} for a more detailed discussion).
\end{enumerate}
\end{remark}

The phenomenon of improving from \eqref{20201222eq03} to \eqref{20210220eq02} under the mixing condition is referred as \emph{dissipation enhancement}. The purpose of this paper is to explore such a phenomenon for the advection equation with $p$-Laplacian evolution \eqref{maineq}. To state the main results, we first extend the definition of the linear dissipation time to its non-linear counterpart.

\begin{definition}
Let $\theta_{0, 0} \in L^2_0(\TT^d)$. The \emph{non-linear dissipation time} associated to the advection equation with $p$-Laplacian diffusion \eqref{maineq} is given by
$$
\kappa_d:= \sup_{s \in \R} \left(\inf \left\{ t-s \Bigg|  t \ge s, \ \textrm{and} \ \norm{\theta_s(t)}_2 \le \frac{\norm{\theta_{s, 0}}_2}{\left[\left(p-2 \right)\|\theta_{s, 0}\|^{p-2}_2 +1 \right]^{\frac{1}{p-2}}} \right\} \right), 
$$
where $\theta_{s}(t)$ is the \emph{weak solution} of \eqref{maineq} with initial data $\theta_{s, 0}$. 
\end{definition}

We make several remarks before we proceed. 

\begin{remark}
\begin{enumerate}
\item[1.] One can find that due to the scaling of the p-Laplacian equation, we can not expect the exponential decay for $L^2$ norm of the solution. This is why we can only consider such non-linear dissipation time.

    \medskip
    
    \item [2.] The study of $\kappa_d$ is a little bit subtle compared to $\tau_d$. More precisely, the solutions of the linear advection equation \eqref{20201222eq01} possess strong or even classical solution, if initial data is provided smooth enough; while, to our best knowledge, one cannot expect the existence of classical solution to~\eqref{maineq}, due to the degeneration or singularity at the points where $\abs{\grad\theta}=0$, and hence the solutions of \eqref{maineq} are interpreted only in the weak sense (see, e.g.,\cite{aronson1986porous,dibenedetto2012degenerate}). We will overcome this difficulty and study the corresponding \emph{dissipation enhancement} phenomenon to~\eqref{maineq} by arguing that such weak solutions (see Definition \ref{weaksoldefn}) have certain regularity (see Section \ref{20210223sec01}); 
    
    \medskip
    
    \item [3.] The definition of $\kappa_d$ is indeed consistent with $\tau_d$. For example, by a standard non-linear Gronwall type estimate, one can see that if $u$ is assumed to be incompressible, then 
    \begin{equation} \label{20210222eq01}
    \kappa_d \le \frac{1}{\nu \lambda_1^{\frac{p}{2}}}
    \end{equation}
    (see Corollary \ref{20210223cor01}). Moreover, if $u$ is assumed to be mixing, then our main results (see Theorem \ref{mainthm01} and Theorem \ref{mainthm02}) suggest that 
    $$
    \kappa_d \le \frac{C}{\nu \lambda_N^{\frac{p}{2}}}
    $$
    for some $C$ and $N$ only depending on \eqref{maineq} and the mixing condition. These facts clearly resemble and generalize both \eqref{20201222eq03} and \eqref{20210220eq02};
    
    \medskip
    
    \item [4.] It is not hard to see that the term $\|\theta_s(t)\|_2$ obeys a polynomial decay with the magnitude $(t-s)^{-\frac{1}{p-2}}$ when $t$ is sufficiently large (see, e.g., replace $t_k'$ by $t$ in \eqref{20210121eq02}). Heuristically, the non-linear dissipation time $\kappa_d$ describes quantitatively how large the coefficient of $(t-s)^{\frac{1}{p-2}}$ is in such a decay, under various assumptions on $u$, this is of the same flavor when we consider the linear dissipation time $\tau_d$, which is used to study how large the coefficient of the term $t-s$ is in the exponential decay (see \eqref{20200221eq01}). 
\end{enumerate}
\end{remark}

We are ready to statement our main results. 

\begin{theorem} \label{mainthm01}
Let $0<\alpha \le 1$, $\beta>0$, $p>2$, $\theta_{0, 0} \in L^2_0(\TT^d)$ and $h: [0, \infty) \to (0, \infty)$ be a strictly decreasing function that vanishes at infinity. If $u$ is strongly $\alpha, \beta$ mixing with rate function $h$, then 
\begin{equation} \label{20210121eq13}
\kappa_d \le \frac{C}{\nu H_1(\nu)^{\frac{p}{2}}{\mathcal H}_{1, \nu, h}}, 
\end{equation} 
where $C$ is an absolute constant that only depends on $h$, $\|\nabla u\|_\infty$, $p$, $\|\theta_{0, 0}\|_2$, the strongly mixing condition, any dimension constants and any constants that depends on $\TT^d$, 
\begin{equation} \label{20210121eq15}
{\mathcal H}_{1, \nu, h}:=\min\left\{1, \ 2^{-p-1} \cdot h^{-1} \left( \frac{H_1(\nu)^{-\frac{\alpha+\beta}{2}}}{2^{1-\frac{\alpha+\beta}{2}}} \right)^{\frac{p-2}{2}} \right\}, 
\end{equation} 
and 
$H_1: (0, \infty) \to (0, \infty)$ is defined by 
\begin{equation} \label{20210121eq14} 
H_1(\nu):=\sup \left\{ \lambda \ \Bigg | \ \frac{\lambda^{\frac{p}{2}} d^{\frac{p-2}{2}} D_p \|\theta_{0, 0}\|_2^{p-2}}{h^{-1} \left( \frac{1}{2} \lambda^{-\frac{\alpha+\beta}{2}}\right)} \cdot e^{4\|\nabla u\|_\infty h^{-1} \left(\frac{1}{2} \lambda^{-\frac{\alpha+\beta}{2}} \right)} \le \frac{\|\nabla u\|_\infty^2}{4\nu} \right\}, 
\end{equation}
Here 
\begin{equation} \label{20201224eq04}
D_p:=48^{p-1}p^p2^{p(p-1)}.
\end{equation} 
and $h^{-1}$ is the inverse function of $h$. 
\end{theorem}

\begin{corollary} \label{strongcor}
Let $\alpha, \beta, u, h, p$ and $\theta_{0, 0}$ be as in Theorem \ref{mainthm01}. Let further, $\nu \ll 1$. Then
\begin{enumerate}
    \item [(1)] If the mixing rate function $h: (0, \infty) \to (0, \infty)$ is the power law
    \begin{equation} \label{powerlaw}
    h(t)=\frac{c}{t^q}
    \end{equation}
    for some $q>0$, then the nonlinear dissipation time is bounded by 
    \begin{equation} \label{corres1}
    \kappa_d \le \frac{C}{\nu \abs{\ln \nu}^{\delta}}, \quad \textrm{where} \ \delta:=\frac{pq}{\alpha+\beta}
    \end{equation}
    and $C=C(\alpha, \beta, c, q, p, d, \|\nabla u\|_{\infty}, \|\theta_{0, 0}\|_2)$ is a finite constant; 
    
    \medskip
    
    \item [(2)] If the mixing rate function $h:[0, \infty) \to (0, \infty)$ is the exponential function 
    \begin{equation} \label{explaw}
    h(t)=c_1\exp(-c_2 t),
    \end{equation} 
    for some constant $c_1, c_2>0$, then the nonlinear dissipation time is bounded by 
    \begin{equation} \label{corres2}
    \kappa_d \le \frac{C}{\nu^{\delta}}, \quad \textrm{where} \ \delta:=\frac{4\|\nabla u\|_{\infty}(\alpha+\beta)}{pc_2+4\|\nabla u\|_{\infty}(\alpha+\beta)}, 
    \end{equation} 
    $C=C(\alpha, \beta, c_1, c_2, p, d, \|\nabla u\|_{\infty}, \|\theta_{0, 0}\|_2)$ is a finite constant. 
\end{enumerate}
\end{corollary}

\begin{remark}
 The strong mixing estimate \eqref{strongmix} has been used by many authors in their research (see  \cites{MathewMezicEA05,LinThiffeaultEA11,Thiffeault12,IyerKiselevEA14,ZelatiDelgadinoElgindi20}). As an example, an exponential universal mixer with regularity $L^{\infty}(0,\infty;W^{1,r})$ for some $r>2$ is constructed by Elgindi and Zlato\v{s} in \cite{ElgindiZlatos19} (see Section 4 of \cite{ZelatiDelgadinoElgindi20} for a more detailed discussion). Some of such flows can be taken to be smooth locally in time (e.g., the flow in \cite{YaoZlatos17}). That means for any $T<\infty$, one could find such flow in the space $L^{\infty}(0,T;W^{1,\infty}(\T^d))$. However, the velocity field $u\in L^{\infty}(0,\infty;W^{1,\infty}(\T^d))$ that mixes arbitrary smooth initial date $f$ exponentially fast remains an open problem. There are two kinds of flows in $L^{\infty}(0,\infty; W^{1,\infty})$ with such mixing property, one is from \cite{YaoZlatos17}, with the flow depending on the solution, another one are so-called relaxation enhancing flow, defined in \cite{ConstantinKiselevEA08}, which are time-independent but without explicit rate $h(t)$, one may find more concrete examples in it. And in \cite{AlbertiCrippaEA19}, the authors construct smooth flows that can mix some specific initial data exponentially fast. However, it cannot be applied to our main theorem, which requires exponential mixing flow for arbitrary initial data.
\end{remark}

\begin{theorem} \label{mainthm02}
Let $0<\alpha \le 1$, $\beta>0$, $p>2$, $\theta_{0, 0} \in L^2_0(\TT^d)$ and $h: [0, \infty) \to (0, \infty)$ be a strictly decreasing function that vanishes at infinity. If $u$ is weakly $\alpha, \beta$ mixing with rate function $h$, then 
\begin{equation} \label{20210121eq13new}
\kappa_d \le \frac{C}{\nu H_2(\nu)^{\frac{p}{2}}{\mathcal H}_{2, \nu, h}}, 
\end{equation} 
where $C$ is an absolute constant that only depends on $h$, $\|\nabla u\|_\infty$, $p$, $\|\theta_{0, 0}\|_2$, the weakly mixing condition, any dimension constants and any constants that depends on $\TT^d$, 
\begin{equation} \label{20210121eq15new}
{\mathcal H}_{2, \nu, h}:=\min\left\{1, \ 2^{-p-1} \cdot h^{-1} \left( \frac{H_2(\nu)^{-\frac{d+2\alpha+2\beta}{4}}}{2^{1-\frac{d+2\alpha+2\beta}{4}}} \right)^{\frac{p-2}{2}} \right\}. 
\end{equation} 
Here $H_2: (0, \infty) \to (0, \infty)$ is defined by 
\begin{equation} \label{20210121eq14new} 
H_2(\mu):=\sup \left\{ \lambda \ \Vast | \ \frac{\lambda^{\frac{p}{2}} d^{\frac{p-2}{2}} D_p \|\theta_{0, 0}\|_2^{p-2}}{h^{-1} \left( \frac{\lambda^{-\frac{d+2\alpha+2\beta}{4}}}{2 \sqrt{\mathfrak{c}}}\right)} \cdot e^{4\|\nabla u\|_\infty h^{-1} \left(\frac{\lambda^{-\frac{d+2\alpha+2\beta}{4}}}{2 \sqrt{\mathfrak{c}}}\right)} \le \frac{\|\nabla u\|_\infty^2}{4\mu} \right\}, 
\end{equation}
where $D_p$ is defined as in \eqref{20201224eq04} and $\mathfrak{c}=\mathfrak{c}(\TT^d)$ is a finite constant that only depends on $\TT^d$.
\end{theorem}

\begin{corollary} \label{weakcor}
Let $\alpha, \beta, u, h, p$ and $\theta_{0, 0}$ be as in Theorem \ref{mainthm01}. Let further, $\nu \ll 1$. Then \begin{enumerate}
    \item [(1)] If the mixing rate function $h: (0, \infty) \to (0, \infty)$ is the power law
    \eqref{powerlaw},  then the nonlinear dissipation time is bounded by 
    \begin{equation} \label{corres3}
    \kappa_d \le \frac{C}{\nu \abs{\ln \nu}^{\delta}}, \quad \textrm{where} \ \delta:=\frac{2pq}{2\alpha+2\beta+d}
    \end{equation}
    and $C=C(\alpha, \beta, c, p, q,  d, \|\nabla u\|_{\infty}, \|\theta_{0, 0}\|_2)$ is a finite constant; 
    
    \medskip
    
    \item [(2)] If the mixing rate function $h:[0, \infty) \to (0, \infty)$ is the exponential \eqref{explaw}, then the nonlinear dissipation time is bounded by 
    \begin{equation} \label{corres4}
    \kappa_d \le \frac{C}{\nu^{\delta}}, \quad \textrm{where} \ \delta:=\frac{2\|\nabla u\|_{\infty}(d+2\alpha+2\beta)}{pc_2+ 2\|\nabla u\|_{\infty}(d+2\alpha+2\beta)}
    \end{equation} 
    and $C=C(\alpha, \beta, c_1, c_2, p, d, \|\nabla u\|_{\infty}, \|\theta_{0, 0}\|_2)$ is a finite constant. 
\end{enumerate}
\end{corollary}

\begin{remark} \label{20210224rem01}
We will see in the proof of Theorem \ref{mainthm02} that the constant $\mathfrak{c}$ can be determined by Weyl's formula, which states
     \begin{equation} \label{Weylfor}
     \lambda_j \simeq \frac{4\pi \Gamma\left(\frac{d}{2}+1 \right)^{\frac{2}{d}}}{\textrm{Vol}(\TT^d)^{\frac{2}{d}}} j^{\frac{2}{d}}
     \end{equation}
     asymptotically as $j \to \infty$ (see, e.g., \cite{MinakshisundaramPleijel49}). 
Then, for example, we can take
\begin{equation} \label{20210224eq100}
\mathfrak{c}:=(1+\varepsilon) \lim_{j \to \infty} \frac{j}{\lambda_j^{\frac{d}{2}}}=\frac{(1+\varepsilon) \textrm{Vol}(\TT^d)}{\left(4 \pi\right)^{\frac{d}{2}} \Gamma\left(\frac{d}{2}+1 \right)}, 
\end{equation} 
for some fixed $\varepsilon$ which is sufficiently small. 
\end{remark}

\begin{remark}
\begin{enumerate}

    \item [1.] One crucial fact to prove the dissipation enhancement results for the linear advection equation is to make use of the iteration structure for the exponential functions. While for the proof of Theorem \ref{mainthm01} and Theorem \ref{mainthm02}, one novelty is that we will adapt new iteration structure given by certain rational functions (see Lemma \ref{20210121lem02}); 
    
    \medskip

    \item [2.] Note that if we let $p=2$ in both Theorem \ref{mainthm01} and Theorem \ref{mainthm02}, 
    $$
    H_{1, v, h}=H_{2, v, h}=\frac{1}{8}
    $$ 
    and our main improvements \eqref{20210121eq13} and \eqref{20210121eq13new}, respectively will ``partially recover'' the improved dissipation time $\tau_d$ for the linear model \eqref{20201222eq01} with $\alpha=1$ under both strongly mixing condition and weakly mixing condition, respectively (see, e.g., \cite[Theorem 2.16 and Theorem 2.19]{FengIyer19}). So are both Corollary \ref{strongcor} and Corollary \ref{weakcor} (see, e.g., \cite[Corollary 2.17 and Corollary 2.20]{FengIyer19}).
    
    \medskip
    
    \item [3.] Observe that when $\nu$ is sufficiently small, ${\mathcal H}_{1, \nu, h}={\mathcal H}_{2, \nu, h} \equiv 1$. Therefore, for $\nu$ sufficiently small, the estimates \eqref{20210121eq13} and \eqref{20210121eq13new} can be simplified as
    $$
    \kappa_d \le \frac{C}{\nu H_i(\nu)^{\frac{p}{2}}}, \quad i=1,2.  
    $$
\end{enumerate}
\end{remark}


\section{Preliminaries} \label{20210223sec01} 

This section aims to introduce the regularity of the weak solutions to equation \eqref{maineq}. As a classical degenerate second-order parabolic equation (see \cite{dibenedetto2012degenerate} for the guidance of the theory), there is extensive literature devoted to \eqref{maineq}. We limit ourselves by referring to the review paper \cite{Kalashnikov1987SomePO} for providing profound insights to \eqref{maineq} and the paper \cite{AntontsevShmarev09}, whose setting is relatively close to ours. In \cite{AntontsevShmarev09}, the authors studied the Dirichlet problem to \eqref{maineq} under a more general setting. To be self-contained, we will adapt their results to ours. Consequently, we show that the weak solution $\theta_s$ satisfies specific energy inequality associated with \eqref{maineq}, which further implies the trivial estimate \eqref{20210222eq01} for $\kappa_d$.

We start with recalling several basic definitions. Following the notations in \cite{AntontsevShmarev09}, let ${\bf V}(\TT^d)$ be the Banach space given by
$$
{\bf V}\left(\TT^d \right):=\left\{g(x): g \in L^2(\TT^d) \cap W^{1, p}(\TT^d)\right\},
$$
and for any $T>0$, we define the Banach space ${\bf W}\left(\TT^d \times [0, T] \right)$ as 
\begin{eqnarray*}
{\bf W}\left(\TT^d \times [0, T] \right)%
&:=& \bigg\{f: [0, T] \mapsto {\bf V}(\TT^d): f \in L^2\left(\TT^d \times [0, T] \right), \\
&& \qquad \qquad \qquad \qquad \qquad\ |D_i f|^p \in L^1 \left(\TT^d \times [0, T] \right)\bigg\},
\end{eqnarray*}
with the norm 
$$
\|f\|_{{\bf W}\left(\TT^d \times [0, T] \right)}:=\|f\|_{L^2\left(\TT^d \times (0, T)\right)}+\sum_{i=1}^d \|D_i f\|_{L^p\left(\TT^d \times (0, T)\right)}.
$$
here $p>2$ is associated with \eqref{maineq}. 

Finally, we let ${\bf W'} \left(\TT^d \times [0, T] \right)$ be the dual space of ${\bf W} \left(\TT^d \times [0, T] \right)$, that is, the space of all linear functionals over ${\bf W}\left(\TT^d \times [0, T] \right)$. Note that $w \in {\bf W'}\left(\TT^d \times [0, T] \right)$ if and only if
$$
\begin{cases}
w=w_0+\sum\limits_{i=1}^d D_i w_i, \quad w_0 \in L^2 \left(\TT^d \times [0, T] \right), \ w_i \in L^{p'} \left(\TT^d \times [0, T] \right); \\
\\
\forall \phi \in {\bf W}\left(\TT^d \times [0, T] \right), \quad \langle w, \phi \rangle_{\bf W}:=\int_0^T \int_{\TT^d} \left(w_0 \phi+\sum\limits_{i=1}^d w_i D_i \phi \right) dxdt.
\end{cases}
$$

\begin{definition} \label{weaksoldefn}
Given $T>0$, a function $\theta(x, t) \in {\bf W}\left(\TT^d \times [0, T] \right) \cap L^\infty\left(0, T; L^2(\TT^d) \right)$ is called a \emph{weak solution} of \eqref{maineq} if for every test function
$$
\zeta \in {\bf Z}:= \left\{ \eta(z): \eta \in {\bf W}\left( \TT^d \times [0, T] \right) \cap L^\infty \left(0, T; L^2(\TT^d) \right), \ \eta_t \in {\bf W'}\left(\TT^d \times [0, T] \right) \right\},   
$$
and every $t_1, t_2 \in [0, T]$, the following identity holds
\begin{equation} \label{weaksolid}
\int_{t_1}^{t_2} \int_{\TT^d} \sum_{i=1}^d\left[\theta \zeta_t-\left(\nu |D_i \theta|^{p-2} D_i \theta - u_i(t) \cdot \theta \right) D_i \zeta \right] dxdt=\int_{\TT^d} \theta \zeta dx \bigg|_{t_1}^{t_2}. 
\end{equation}
\end{definition}

\begin{theorem}[{\cite[Theorem 3.1]{AntontsevShmarev09}}] \label{regthm}
For every $s \ge 0$ and $\theta_{0, 0} \in L_0^2(\TT^d)$, the problem \eqref{maineq} has at least one weak solution $\theta_s \in  {\bf W}\left(\TT^d \times [s, T] \right)$ satisfying the estimate
\begin{equation}\label{e:energy identity1}
\|\theta_s\|_{L^\infty(s, T; L^2(\TT^d))}^2+\nu\int_s^T \int_{\TT^d }\sum_{i=1}^d\abs{D_i \theta_s}^p dxdt \le C \left(\|\theta_{0, 0}\|_{L^2(\TT^d)}^2+1 \right), 
\end{equation}
where $C$ is an absolute constant independent of $\theta_s$. Moreover, $\frac{d}{dt} \theta_s \in {\bf W'}\left(\TT^d \times [s, T]\right)$. 
\end{theorem}

\begin{remark}\label{rmk:uniqueness}
From Section 5 of \cite{AntontsevShmarev09} or Chapter 6 of \cite{dibenedetto2012degenerate}, one also has the uniqueness of the weak solution in Theorem \ref{regthm}. For simplicity we omit the regularity proofs here.
\end{remark}
Note that Theorem \ref{regthm} and Remark \ref{rmk:uniqueness} suggests \eqref{maineq} admits a unique weak solution, which belongs to the test function space ${\bf Z}$. As some byproducts, we collect several useful facts about \eqref{maineq}.

\begin{corollary} \label{20210223cor01}
\begin{enumerate}
    \item [1.] If $\theta_{0, 0} \in L^2_0(\TT^d)$, then so does  $\theta_{s, 0}:=\theta_0(s)$; 
    
    \medskip
    
    \item [2.] For any $\zeta \in {\bf Z}$, there holds that
    \begin{equation} \label{weakexact}
    \int_{\TT^d} \sum_{i=1}^d\left[\frac{ \partial \theta_s}{\partial t} \cdot \zeta+\left(\nu |D_i \theta_s|^{p-2} D_i \theta_s - u_i(t) \cdot \theta_s \right) \cdot D_i \zeta \right] dx=0;
    \end{equation} 
    
    \medskip
    
    \item [3.] The following estimate holds
    $$
    \kappa_d \le \frac{1}{\nu \lambda_1^{\frac{p}{2}}}; 
    $$

\end{enumerate}
\end{corollary}

\begin{proof}
1. Recall that $\theta_{s, 0}:=\theta_0(s)$. Therefore, it suffices to note that the constant function $\zeta \equiv 1 \in {\bf Z}$. The desired claim follows from letting $t_1=0$ and $t_2=s$ in \eqref{weaksolid}.

\medskip

2. This follows by taking the $t$ derivative on both sides of \eqref{weaksolid} with $t_1=s$ and $t_2=t$, respectively. Then by Fundamental Theorem of Calculus, one has 
$$
\int_{\TT^d} \sum_{i=1}^d\left[\theta \zeta_t-\left(\nu |D_i \theta|^{p-2} D_i \theta - u_i(t) \cdot \theta \right) D_i \zeta \right] dxdt=\int_{\TT^d} \frac{\partial}{\partial t}(\theta \zeta) dx . 
$$
Since $\theta,\zeta\in {\bf W}$ and $\frac{\partial \zeta}{\partial t},\frac{\partial\theta}{\partial t}\in {\bf W}'$, one can get \eqref{weakexact} simply by product rule.
\medskip

3. Since $\theta_s \in {\bf Z}$, we can therefore let $\zeta=\theta_s$ in \eqref{weakexact}. This together with the fact that $u$ is divergence free yields the energy estimate for \eqref{maineq}: 
\begin{equation} \label{20201220eq01}
    \frac{1}{2} \partial_t \|\theta_s(t)\|_2^2+\nu \|\grad \theta_s(t)\|_p^p=0,
\end{equation}
which implies
$$
\partial_t \|\theta_s(t)\|_2^2 \le -2 \nu \|\grad \theta_s(t) \|_2^p \le -2\nu \lambda_1^{\frac{p}{2}} \|\theta_s(t)\|_2^p
$$
Therefore, an easy application of the (non-linear) Gronwall's inequality yields 
\begin{equation} \label{20201220eq02}
\norm{\theta_s(t)}_2 \le \frac{\norm{\theta_{s, 0}}_2}{\left[ \nu\lambda_1^{\frac{p}{2}}\left(p-2 \right)(t-s)\|\theta_{s, 0}\|^{p-2}_2 +1 \right]^{\frac{1}{p-2}}},
\end{equation} 
which implies the desired claim. 
\end{proof}

\section{Dissipation enhancement for evolution $p$-Laplacian advection equations}

In this section we prove Theorem \ref{mainthm01} and \ref{mainthm02}. The main idea behind the proof is to split the analysis into two different cases. In the first case, we assume $\frac{\|\nabla \theta_s(t) \|_p}{\|\theta_s(t)\|_2}$ is large, and obtain decay of $\|\theta_s\|_2$ using the energy inequality \eqref{20201220eq01}; in the second case, $\frac{\|\nabla \theta_s(t) \|_p}{\|\theta_s(t)\|_2}$ is small, and hence the dynamics are well approximated by that of the underlying inviscid dynamics. The mixing assumption now forces the generation of high frequencies, and the rapid dissipation of these gives an enhanced decay of $\|\theta_s\|_2$.

\medskip
\subsection{The Strongly Mixing Case.} Let $s \ge 0$ be any time. We first consider the case when 
\begin{equation} \label{20210119eq03}
\frac{\|\nabla \theta_s(t) \|_p}{\|\theta_s(t)\|_2}
\end{equation}
is large. More precisely, if for some $c_0>0$, we have
$$
\|\nabla \theta_s(t)\|_p \ge c_0^{\frac{1}{2}} \|\theta_s(t)\|_2, \quad \textrm{for all} \ s \le t \le t_0, 
$$
then by the energy inequality \eqref{20201220eq01}, we have
$$
\partial_t \|\theta_s(t)\|_2^2 \le -2\nu c_0^{\frac{p}{2}} \|\theta_s(t)\|_2^p,
$$
which implies 
\begin{equation} \label{20210119eq02}
\|\theta_s(t)\|_2 \le \frac{\|\theta_{s, 0}\|_2}{\left[\nu c_0^{\frac{p}{2}} (p-2)(t-s) \|\theta_{s, 0}\|_2^{p-2}+1 \right]^{\frac{1}{p-2}}}, \quad \textrm{for all} \ s \le t \le t_0. 
\end{equation}

We now turn to the second case, in which, the ratio \eqref{20210119eq03} is relatively ``small''. The idea is to use the strongly mixing assumption of $u$ to show that there exists a moment $t_0>s$, such that \eqref{20210119eq02} still hold. We now turn to the details. 

We start with understanding the relation between our non-linear model \eqref{maineq} and the transport equation. More precisely, we have the following result.

\begin{lemma} \label{20210120lem01}
Let $\phi_s$ be the solution of the following transport equation
\begin{equation} \label{20201224eq01}
\begin{cases}
\partial_t \phi_s+\left(u \cdot \grad \right) \phi_s=0; \\
\\
\phi_s(t)=\theta_{s, 0}, \quad t=s. 
\end{cases}
\end{equation} 
Then for all $t \ge s$, 
$$
\norm{\theta_s(t)-\phi_s(t)}_2^2 \le \frac{d^{\frac{p-2}{2}}D_p \nu}{\norm{\grad u}_{\infty}} \cdot e^{2\norm{\grad u}_{\infty}(t-s)} \cdot \norm{\grad \theta_{s, 0}}_p^p,
$$
where $D_p$ is defined in \eqref{20201224eq04}.
\end{lemma}

\begin{remark}\label{rmk:diff estimate}
Lemma \ref{20210120lem01} is parallel to Lemma 5.2 in \cite{FengIyer19}, which shows the distance between the viscous and inviscid problem grows at most exponentially in time. For the study of the linear problem \eqref{eq:AD}, the authors of \cite{ZelatiDelgadinoElgindi20} developed nice tricks, by estimating $\int_0^T\norm{\lap\vartheta}_2^2 dt$ from above, to capture the polynomial growing of that distance (see equation (2.12) in \cite{ZelatiDelgadinoElgindi20}). They, therefore, obtained a faster dissipation rate than \cite{FengIyer19} (see Remark 2.18 in \cite{FengIyer19} for a more detailed discussion). After checking, we find the tricks developed in \cite{ZelatiDelgadinoElgindi20} can be formally apply to \eqref{eq:AD2}. However, in this paper, we only consider the weak solution to \eqref{eq:AD2}, by the Definition \ref{weaksoldefn}, we do not have enough regularity to carry out that method rigorously. More precisely, the expression $\int_0^T\norm{\grad \cdot \left(\abs{\grad \theta }^{p-2} \grad \theta \right)}_2^2 dt$ is not well defined for \eqref{eq:AD2}. Considering this issue, we use weak formulation to recover the 
estimate in \cite{FengIyer19}.
\end{remark}
\begin{proof}
Let $\omega(t)=\theta_s(t)-\phi_s(t)$. Note that $\omega(s)=0$. Since the solution for the transport equation \eqref{20201224eq01} exists in strong sense, this means
$$
\int_{\TT^d} \left[ \frac{\partial \phi_s}{\partial t}\cdot \zeta - \sum_{i=1}^d u_i(t) \cdot \phi_s \cdot D_i \zeta \right] dx=0
$$
for any $\zeta \in {\bf Z}$. We now subtract the above equation with \eqref{weakexact} (with the same choice of $\zeta$ in both equations) to get 
\begin{eqnarray} \label{20210224eq01}
&& \int_{\TT^d} \frac{\partial \omega}{\partial t} \cdot \zeta+ \sum_{i=1}^d \left(\nu\abs{D_i \omega}^{p-2} D_i \omega - u_i(t) \cdot \omega(t) \right) \cdot D_i \zeta dx \nonumber \\
&& \quad \quad \quad = \nu \int_{\TT^d}\sum_{i=1}^d \abs{D_i \omega}^{p-2} D_i \omega \cdot D_i \zeta dx-\nu \int_{\TT^d} \sum_{i=1}^d \abs{D_i \theta}^{p-2} D_i \theta \cdot D_i \zeta dx, 
\end{eqnarray}
for any $\zeta \in {\bf Z}$. Observe that $\omega \in {\bf Z}$. Indeed, it suffices to show that $\phi_s \in {\bf Z}$, which can be easily verified by Gronwall's inequality. Therefore, we are able to take $\zeta=\omega$ in \eqref{20210224eq01}, which further gives
\begin{eqnarray}\label{e:energy identity2}
&&\frac{1}{2\nu} \partial_t \|\omega\|_2^2+\|\grad \omega\|_p^p \\
&& = \left\langle \left|\grad \omega \right|^{p-2} \grad \omega-\left|\grad \theta_s \right|^{p-2} \grad \theta_s, \grad \omega \right\rangle \nonumber\\
&&= \left\langle \left|\grad \omega \right|^{p-2} \grad \omega-\left|\grad \left(\omega+\phi_s \right) \right|^{p-2} \grad \left(\omega+\phi_s \right), \grad \omega \right\rangle \nonumber\\
&&= I_1+I_2, \nonumber
\end{eqnarray}
where 
$$
I_1:=-\left\langle \left|\grad \omega \right|^{p-2} \grad \phi_s, \grad \omega \right\rangle
$$
and
$$
I_2:= \left\langle \left(\left|\grad \omega \right|^{p-2}-\left|\grad (\omega+\phi_s)\right|^{p-2}\right) \grad (\omega+\phi_s), \grad \omega \right\rangle. 
$$

\textit{Estimate of $I_1$.}

\begin{equation} \label{20201223eq07}
I_1\le \int_{\TT^d} \left|\grad \omega \right|^{p-1} |\grad \phi_s| dx\le \|\grad \omega\|_p^{{p-1}} \|\grad \phi_s\|_p. 
\end{equation}

\textit{Estimate of $I_2$.} Note that 
\begin{eqnarray} \label{20201223eq04}
I_2%
&\le&  \int_{\TT^d} \left|\left|\grad \omega \right|^{p-2}-\left|\grad (\omega+\phi_s)\right|^{p-2} \right| \cdot \left| \grad( \omega +\phi_s) \right| \left| \grad \omega \right| dx\nonumber \\
&=& (p-2) \int_{\TT^d} \left(\alpha_{t, x}|\grad \omega|+(1-\alpha_{t, x}) \left|\grad (\omega+\phi_s) \right| \right)^{p-3}  \nonumber\\
&& \quad \quad \quad \quad \quad \quad \quad \quad \quad \left| \left|\grad \omega \right|-\left|\grad\left(\omega+\phi_s \right) \right| \right| \cdot \left| \grad( \omega +\phi_s) \right| \left| \grad \omega \right| dx\nonumber \\
&\le& (p-2)  \int_{\TT^d} \left(\alpha_{t, x}|\grad \omega|+(1-\alpha_{t, x}) \left|\grad (\omega+\phi_s) \right| \right)^{p-3} \\
&& \quad \quad \quad \quad \quad \quad \quad \quad \quad \quad \quad \quad \quad \left|\grad \phi_s \right| \cdot \left| \grad( \omega +\phi_s) \right| \left| \grad \omega \right| dx, \nonumber
\end{eqnarray}
where in the second equation above, we use the mean value theorem and $\alpha_{t, x} \in [0, 1]$ which depends on the values of $t$ and $x \in \TT^d$. We now consider two different cases for the value of $p$. 

\medskip

\textit{Case I: $p \ge 3$.} Recall for any $\ell \ge 0$ and $a, b>0$, we have
\begin{equation} \label{20201223eq06}
(a+b)^\ell \le C_\ell (a^\ell+b^\ell), 
\end{equation}
where
$$
C_\ell=
\begin{cases}
1, \quad \quad \quad 0 \le \ell \le 1; \\
2^{\ell-1}, \quad \quad \quad \ell>1. 
\end{cases}
$$
Therefore, by \eqref{20201223eq06}, 
\begin{eqnarray} \label{20201223eq05}
\eqref{20201223eq04}%
&\le& (p-2)C_{p-3} \cdot \int_{\TT^d} \left|\grad \omega \right|^{p-2} \left| \grad \phi_s \right| \left|\grad( \omega+\phi_s)\right|dx \nonumber \\
&& \quad + (p-2)C_{p-3} \cdot \int_{\TT^d} \left|\grad(\omega+\phi_s) \right|^{p-2} \left|\grad \phi_s \right| \left|\grad \omega \right| dx
\end{eqnarray}
It is easy to check that
$$
2(p-2)C_{p-3}+2(p-2)C_{p-3}C_{p-2} \le p2^p, 
$$
and hence, we have 
\begin{eqnarray} \label{20201223eq08}
&&\textrm{RHS of} \ \eqref{20201223eq05} \nonumber\\
&&\le  p2^p \cdot \bigg[\int_{\TT^d} \left|\grad \omega \right|^{p-1} \left|\grad \phi_s \right| dx+\int_{\TT^d} \left|\grad \omega \right|^{p-2} \left|\grad \phi_s \right|^2 dx \nonumber \\
&& \quad \quad \quad \quad \quad \quad \quad \quad \quad \quad \quad \quad \quad \quad \quad + \int_{\TT^d} \left| \grad \phi_s \right|^{p-1} \left|\grad \omega \right| dx \bigg] \nonumber \\
&& \le  p2^p \cdot \left(\|\grad \omega \|_p^{p-1}\|\grad \phi_s\|_p+\|\grad \phi_s \|_p^{p-1}\|\grad \omega\|_p+\|\grad \omega\|_p^{p-2} \|\grad \phi_s \|_p^2 \right) 
\end{eqnarray}

\medskip

\textit{Case II: $2<p<3$.} The second case is similar to the first one. The only difference is how we estimate the term 
\begin{equation} \label{20201223eq09}
\left(\alpha_{t, x}|\grad \omega|+(1-\alpha_{t, x}) \left|\grad (\omega+\phi_s) \right| \right)^{p-3}.
\end{equation} 
Note that since $p-3<0$, it follows that 
$$
\eqref{20201223eq09} \le \min\left\{\left(\alpha_{t, x}|\grad \omega| \right)^{p-3}, \left((1-\alpha_{t, x}) \left|\grad (\omega+\phi_s) \right| \right)^{p-3} \right\}. 
$$
An easy pigeonholing yields for each $t$ and $x$, at least one of $\alpha_{t, x}$ and $1-\alpha_{t, x}$ belongs to $\left[\frac{1}{2}, 1 \right]$, this allows us to bound \eqref{20201223eq09} further by
$$
2^{3-p} \left(\left|\grad \omega \right|^{p-3}+\left|\grad(\omega+\theta_s) \right|^{p-3} \right) \le 8\left(\left|\grad \omega \right|^{p-3}+\left|\grad(\omega+\theta_s) \right|^{p-3} \right).
$$
This implies when $2<p<3$, we have
\begin{eqnarray} \label{20201223eq10}
\eqref{20201223eq04}%
&\le& 8 (p-2) \bigg[\int_{\TT^d} \left|\grad \omega \right|^{p-2} \left|\grad \phi_s \right| \left|\grad(\omega+\phi_s) \right| dx \nonumber \\
&& \quad \quad \quad \quad \quad \quad \quad \quad+\int_{\TT^d} \left|\grad \omega \right|^{p-2} \left| \grad \phi_s \right| \left|\grad( \omega+\phi_s) \right| dx \bigg] \nonumber \\
&\le& 16  \cdot \left(\|\grad \omega \|_p^{p-1}\|\grad \phi_s\|_p+\|\grad \phi_s \|_p^{p-1}\|\grad \omega\|_p+\|\grad \omega\|_p^{p-2} \|\grad \phi_s \|_p^2 \right). 
\end{eqnarray}

\medskip

Therefore, combining both cases, namely \eqref{20201223eq08} and \eqref{20201223eq10}, we have
\begin{equation} \label{20201223eq11}
I_2 \le 16p2^p \cdot \left(\|\grad \omega \|_p^{p-1}\|\grad \phi_s\|_p+\|\grad \phi_s \|_p^{p-1}\|\grad \omega\|_p+\|\grad \omega\|_p^{p-2} \|\grad \phi_s \|_p^2 \right).
\end{equation}

\medskip

By \eqref{20201223eq07}, \eqref{20201223eq11} and a standard calculation by using Young's inequality, there holds 
$$
\frac{1}{2\nu} \partial_t \|\omega\|_2^2+\|\grad \omega\|_p^p \le I_1+I_2 \le \|\grad \omega \|_p^p+D_p \|\grad \phi_s \|_p^p,
$$
where $D_p$ is defined in \eqref{20201224eq04}. Hence, we have
\begin{equation} \label{20201224eq02}
\partial_t \|\omega\|_2^2 \le 2D_p \nu \|\grad \phi_s\|_p^p, 
\end{equation}
Since $\phi_s$ solves the transport equation \eqref{20201224eq01}, an easy application of the Gronwall's inequality yields
\begin{equation} \label{20201224eq03}
\left\|\grad \phi_s\right \|^p_p \le d^{\frac{p-2}{2}} \cdot e^{2\|\grad u\|_{\infty}(t-s)}\|\grad \theta_{s, 0}\|_p^p.
\end{equation} 

The desired result then follows from integrating \eqref{20201224eq02} from $s$ to $t$ and use the inequality \eqref{20201224eq03}. 
\end{proof}
The next lemma deals with the case when the initial data has a relatively ``small'' $W^{1, p}$ energy. 


\begin{lemma} \label{20210121lem02}
Choose $\lambda_N$ to be the largest eigenvalue satisfying $\lambda \le H_1(\nu)$ where $H_1(\nu)$ is defined as in \eqref{20210121eq14}. If
\begin{equation} \label{20201224eq05}
\|\grad \theta_{s, 0} \|_p<\lambda_N^{\frac{1}{2}} \|\theta_{s, 0}\|_2
\end{equation} 
then we have 
\begin{equation} \label{20201224eq06}
\|\theta_s(t_0)\|_2 \le \frac{\|\theta_{s, 0}\|_2}{\left(\frac{\nu H_1(\nu)^{\frac{p}{2}}}{2^{\frac{3p}{2}+1}} h^{-1} \left(\frac{H_1(\nu)^{-\frac{\alpha+\beta}{2}}}{2^{1-\frac{\alpha+\beta}{2}}} \right)^{\frac{p-2}{2}} (p-2)(t_0-s)\|\theta_{s, 0}\|_2^{p-2}+1 \right)^{\frac{1}{p-2}}}
\end{equation} 
at a time $t_0$ given by
$$
t_0:=s+2h^{-1} \left(\frac{\lambda_N^{-\frac{\alpha+\beta}{2}}}{2}  \right)
$$
\end{lemma}

\begin{proof}
Integrating the energy inequality \eqref{20201220eq01}, we have
\begin{equation} \label{20210120eq06}
\|\theta_s(t)\|_2^2=\|\theta_{s, 0}\|_2^2-2\nu \int_s^t \left\|\grad \theta_s(r) \right\|_p^p dr
\end{equation}
for any $t>s$. 

We claim that for the choice of $\lambda_N$ and $t_0$ will guarantee
\begin{equation} \label{20201226eq01}
\int_s^{t_0} \|\grad \theta_s(r)\|_p^pdr \ge\left( \frac{\lambda_N(t_0-s)}{8} \right)^{\frac{p}{2}} \cdot \|\theta_{s, 0}\|_2^p.
\end{equation}
Since $p>2$, H\"older's inequality yields
$$
\int_s^{t_0} \|\grad \theta_s(r)\|_p^pdr \ge \left(\int_s^{t_0} \|\grad \theta_s(r)\|_2^2dr \right)^{\frac{p}{2}}. 
$$
We now establish a lower bound of the term 
$$
\int_s^{t_0} \|\grad \theta_s(r)\|_2^2 dr
$$
by following the argument in \cite[Lemma 5.1]{FengIyer19}. More precisely, 
\begin{eqnarray} \label{20201226eq03}
\int_s^{t_0} \|\grad \theta_s (r) \|_2^2 dr %
&\ge& \lambda_N \int_{\frac{t_0+s}{2}}^{t_0} \|(I-P_N)\theta_s(r)\|_2^2dr \nonumber \\
&\ge& \frac{\lambda_N}{2} \int_{\frac{t_0+s}{2}}^{t_0} \|(I-P_N)\phi_s(r)\|_2^2dr \nonumber \\
&& \quad \quad \quad -\lambda_N \int_{\frac{t_0+s}{2}}^{t_0} \|(I-P_N)(\theta_s(r)-\phi_s(r))\|_2^2dr \nonumber \\
&\ge& \frac{\lambda_N(t_0-s)}{4}\|\theta_{s, 0}\|_2^2-\frac{\lambda_N}{2} \int_{\frac{t_0+s}{2}}^{t_0} \|P_N \phi_s(r)\|_2^2dr \\
&& \quad \quad \quad -\lambda_N \int_{\frac{t_0+s}{2}}^{t_0} \|\theta_s(r)-\phi_s(r)\|_2^2dr,  \nonumber
\end{eqnarray}
where $P_N$ is the projection operator from $L^2$ to the subspaces spanned by the first $N$ eigvenvectors. Therefore, it suffices to bound the last two terms in \eqref{20201226eq03}. For the second term, using the strongly mixing condition \eqref{strongmix} and \eqref{20201224eq05}, we see that
\begin{eqnarray} \label{20210120eq01}
\int_{\frac{t_0+s}{2}}^{t_0} \|P_N\phi_s(r)\|_2^2dr%
&\le& \lambda_N^\beta \int_{\frac{t_0+s}{2}}^{t_0} \|\phi_s(r)\|_{\dot{H}^{-\beta}}^2dr \le \lambda_N^\beta \int_{\frac{t_0+s}{2}}^{t_0} h(r-s)^2\|\theta_{s, 0}\|_{\dot{H}^{\alpha}}^2dr \nonumber \\
&\le& \frac{\lambda_N^\beta(t_0-s)}{2} h\left(\frac{t_0-s}{2} \right)^2 \|\theta_{s, 0}\|_{\dot{H}^\alpha}^2 \nonumber\\
&\le& \frac{\lambda_N^\beta(t_0-s)}{2} h\left(\frac{t_0-s}{2} \right)^2 \| \theta_{s, 0}\|_2^{2-2\alpha} \|\grad \theta_{s, 0}\|_2^{2\alpha}\nonumber\\
&\le& \frac{\lambda_N^\beta(t_0-s)}{2} h\left(\frac{t_0-s}{2} \right)^2 \| \theta_{s, 0}\|_2^{2-2\alpha} \|\grad \theta_{s, 0}\|_p^{2\alpha}\nonumber\\
&\le& \frac{\lambda_N^{\alpha+\beta}(t_0-s)}{2} h\left(\frac{t_0-s}{2} \right)^2 \| \theta_{s, 0}\|_2^2. 
\end{eqnarray}
Finally, we bound the last term \eqref{20201226eq03}. By Lemma \ref{20210120lem01}, we have
\begin{eqnarray} \label{20210120eq02}
\int_{\frac{t_0+s}{2}}^{t_0} \|\theta_s(r)-\phi_s(r)\|_2^2dr %
&\le& \frac{d^{\frac{p-2}{2}}D_p \nu}{2\|\nabla u \|_\infty^2} \cdot e^{2\|\nabla u\|_\infty (t_0-s)} \|\nabla \theta_{s, 0}\|_p^p \nonumber \\  
&\le& \frac{\lambda_N^{\frac{p}{2}}\cdot d^{\frac{p-2}{2}}D_p \nu}{2\|\nabla u\|_\infty^2} e^{2\|\nabla u\|_\infty (t_0-s)} \|\theta_{s, 0}\|_2^p \nonumber \\
&\le& \frac{\lambda_N^{\frac{p}{2}} \cdot d^{\frac{p-2}{2}} D_p \nu \|\theta_{0, 0}\|_2^{p-2}}{2\|\nabla u\|_\infty^2} \cdot  e^{2\|\nabla u\|_\infty (t_0-s)} \|\theta_{s, 0}\|_2^2. 
\end{eqnarray}
Therefore, combining \eqref{20201226eq03} with \eqref{20210120eq01} and \eqref{20210120eq02}, we have 
\begin{eqnarray} \label{20210120eq03}
&& \int_s^{t_0} \|\nabla \theta_s(r)\|_2^2 dr \ge \lambda_N(t_0-s) \|\theta_{s, 0}\|_2^2 \cdot \bigg(\frac{1}{4}-\frac{\lambda_N^{\alpha+\beta}}{4}h\left(\frac{t_0-s}{2} \right)^2 \nonumber \\
&& \quad \quad \quad \quad  \quad \quad  \quad \quad  \quad \quad \quad \quad  -\frac{\lambda_N^{\frac{p}{2}} \cdot d^{\frac{p-2}{2}} D_p \nu \|\theta_{0, 0}\|_2^{p-2}}{2\|\nabla u\|_\infty^2(t_0-s)} \cdot  e^{2\|\nabla u\|_\infty (t_0-s)} \bigg)
\end{eqnarray}
By our choice of $\lambda_N$ and $t_0$, we have
$$
\frac{\lambda_N^{\alpha+\beta}}{4} h \left(\frac{t_0-s}{2} \right)^2 \le \frac{1}{16}, \quad \textrm{and} \quad \frac{\lambda_N^{\frac{p}{2}} \cdot d^{\frac{p-2}{2}} D_p \nu \|\theta_{0, 0}\|_2^{p-2}}{2\|\nabla u\|_\infty^2(t_0-s)} \cdot  e^{2\|\nabla u\|_\infty (t_0-s)} \le \frac{1}{16}. 
$$
Hence, \eqref{20210120eq03} reduces to
\begin{equation} \label{20210224eq23}
\int_s^{t_0}\|\nabla \theta_s(r) \|_2^2 dr \ge \frac{\lambda_N(t_0-s) \|\theta_{s, 0}\|_2^2}{8},
\end{equation} 
which further implies \eqref{20201226eq01}.

\medskip

We now turn to the proof of \eqref{20201224eq06}. To begin with, we collect several facts: 
\begin{enumerate}
    \item [1.] Recall that we are only interested in the case when $\nu$ is small, this together with the definition of $H_1(\nu)$ asserts that for $\nu$ sufficiently small, $H_1(\nu)$ is increasing when $\nu$ tends to $0$;
    
    \medskip

    \item [2.] From the definition of $H_1(\nu)$, together with the fact that $h^{-1}$ is strictly decreasing (since $h$ is assumed to be strictly deceasing), one can check that
    $$
    \lim_{\nu \to 0} \nu \lambda_N^{\frac{p}{2}} h^{-1} \left( \frac{\lambda_N^{-\frac{\alpha+\beta}{2}}}{2} \right)^{\frac{p}{2}}=0. 
    $$
    Moreover, this together with the definition of $t_0$, suggests that 
    \begin{equation} \label{20210120eq04}
       \lim_{\nu \to 0} \nu \lambda_N^{\frac{p}{2}} h^{-1} \left( \frac{\lambda_N^{-\frac{\alpha+\beta}{2}}}{2} \right)^{\frac{p-2}{2}}(t_0-s)=0;
    \end{equation}

     \medskip
     
     \item [3.] When $\nu$ is sufficiently small, there holds that
     \begin{equation} \label{20210120eq05}
     \frac{1}{2} H_1(\nu) \le \lambda_N \le H_1(\nu). 
     \end{equation} 
     This is due to Fact 1 above and the Weyl's lemma \eqref{Weylfor}.

\end{enumerate}

\medskip

Combining both \eqref{20210120eq06} and \eqref{20201226eq01}, we have
\begin{eqnarray} \label{20210121eq01}
\|\theta_s(t_0)\|_2^2 \nonumber %
&\le& \|\theta_{s, 0}\|_2^2-2\nu \left( \frac{\lambda_N(t_0-s)}{8} \right)^{\frac{p}{2}} \cdot \|\theta_{s, 0}\|_2^p \nonumber \\
&=& \|\theta_{s, 0}\|_2^2 \cdot \left[1- 2\nu \left( \frac{\lambda_N(t_0-s)}{8} \right)^{\frac{p}{2}} \cdot \|\theta_{s, 0}\|_2^{p-2} \right] \nonumber \\
&=& \|\theta_{s, 0}\|_2^2 \cdot \left[1- 2\nu \cdot \frac{\lambda_N^\frac{p}{2}}{8^{\frac{p}{2}}} \cdot (t_0-s)^{\frac{p-2}{2}} \cdot (t_0-s) \cdot \|\theta_{s, 0}\|_2^{p-2} \right] \nonumber \\
&=&  \|\theta_{s, 0}\|_2^2 \cdot \left[1- \nu \cdot \frac{\lambda_N^\frac{p}{2}}{4^{\frac{p}{2}}} \cdot h^{-1} \left( \frac{\lambda_N^{-\frac{\alpha+\beta}{2}}}{2} \right)^{\frac{p-2}{2}} \cdot (t_0-s) \cdot \|\theta_{s, 0}\|_2^{p-2} \right].
\end{eqnarray}
Note that $\|\theta_{s, 0}\|_2 \le \|\theta_{0, 0}\|_2$, by \eqref{20210120eq04}, we may assume that when $\nu$ is sufficiently small,
$$
0<\nu \cdot \frac{\lambda_N^\frac{p}{2}}{4^{\frac{p}{2}}} \cdot h^{-1} \left( \frac{\lambda_N^{-\frac{\alpha+\beta}{2}}}{2} \right)^{\frac{p-2}{2}} \cdot (t_0-s) \cdot \|\theta_{s, 0}\|_2^{p-2} <\frac{1}{2}.
$$

On the other hand, by Taylor expansion, for $0<x< 1$ and $p>2$ we have 
$$
1-\frac{2x}{p-2} \le \frac{1}{(1+x)^{\frac{2}{p-2}}}.
$$
Apply the above fact we further get 
$$
\eqref{20210121eq01} \le \frac{\|\theta_{s, 0}\|_2^2}{\left(\frac{\nu \lambda_N^{\frac{p}{2}}}{2^{p+1}} h^{-1} \left(\frac{\lambda_N^{-\frac{\alpha+\beta}{2}}}{2} \right)^{\frac{p-2}{2}} (p-2)(t_0-s)\|\theta_{s, 0}\|_2^{p-2}+1 \right)^{\frac{2}{p-2}}},
$$
which, together with \eqref{20210120eq05}, implies \eqref{20201224eq06}. The proof is complete. 
\end{proof}

The following lemma is the key ingredient to run the iteration argument when we prove Theorem \ref{mainthm01}. 


\begin{lemma} \label{20210121lem021}
For $a>0$ and $p>2$, then the function
\begin{equation}\label{e:fcn Fa}
F_a(x):=\frac{x}{(ax^{p-2}+1)^{\frac{1}{p-2}}}
\end{equation}
is increasing for $x \ge 0$.
Further, if 
\begin{equation} \label{20210121eq21}
    x_1\le F_{b(t'_1-t'_0)}(x_0),
\end{equation}
and
\begin{equation} \label{20210121eq22}
    x_2\le F_{c(t'_2-t'_1)}(x_1),
\end{equation}
where $b,c>0$, $x_0, x_1, x_2>0$, and $t'_2>t'_1>t'_0$. Then we have,
\begin{equation*}
    x_2\le F_{d(t'_2-t'_0)}(x_0), 
\end{equation*}
where $d=\min\{b, c\}$. 
\end{lemma}
\begin{proof}
It is easy to check that $F'_a(x)=\frac{1}{(ax^{p-2}+1)^{\frac{p-1}{p-2}}}>0$, which clearly yields the first claim. And by a direct computation we have
\begin{eqnarray}
    x_2
    &\le& F_{c(t'_2-t'_1)}(F_{b(t'_1-t'_0)}(x_0)) \nonumber  \\
    &=& \frac{x_0}{\big(c(t'_2-t'_1)x_{0}^{p-2}+b(t'_1-t'_0)x_{0}^{p-2}+1\big)^{1/p-2}} \nonumber \\
    &\le& \frac{x_0}{\big(d(t'_2-t'_0)x_{0}^{p-2}+1\big)^{1/p-2}} = F_{d(t'_2-t'_0)}(x_0)\nonumber 
\end{eqnarray}
\end{proof}
Finally, let us turn to prove the main result Theorem \ref{mainthm01}. 
\begin{proof} [Proof of Theorem \ref{mainthm01}.]
Repeatedly applying  \eqref{20210119eq02} with $c_0=\lambda_N$ and Lemma \ref{20210121lem02}, together with the fact \eqref{20210120eq05} and Lemma \ref{20210121lem021},  we obtain an increasing sequence of times $(t_k')$, such that for each $k \ge 1$, there holds 
\begin{equation} \label{20210121eq02}
    \|\theta_s(t_k')\|_2 \le \frac{\|\theta_{s, 0}\|_2}{\left( \frac{\nu H_1(\nu)^{\frac{p}{2}} {\mathcal H}_{1, \nu, h}}{2^{\frac{p}{2}}} \cdot (p-2)(t_k'-s) \|\theta_{s, 0}\|_2^{p-2}+1 \right)^{\frac{1}{p-2}}},
\end{equation} 
where ${\mathcal H}_{1, \nu, h}$ is defined in \eqref{20210121eq15}, that is
$$
{\mathcal H}_{1, \nu, h}:=\min\left\{1, \ 2^{-p-1} \cdot h^{-1} \left( \frac{H_1(\nu)^{-\frac{\alpha+\beta}{2}}}{2^{1-\frac{\alpha+\beta}{2}}} \right)^{\frac{p-2}{2}} \right\}. 
$$
and
$$
t_{k+1}'-t_k' \le t_0. 
$$
Let us include more details for the iteration argument above. Indeed, although our estimates \eqref{20210119eq02} and \eqref{20201224eq06} are no longer linear, the iteration structure still works. For example, let $t'_0:=s$ and suppose
\begin{equation} \label{20210121eq03}
\|\theta_s(t'_1)\|_2 \le \frac{\|\theta_{s, 0}\|_2}{\left(\frac{\nu H_1(\nu)^{\frac{p}{2}}}{2^{\frac{p}{2}}}(p-2)(t'_1-s) \|\theta_{s, 0}\|_2^{p-2}+1 \right)^{\frac{1}{p-2}}}
\end{equation} 
and
\begin{equation} \label{20210121eq04}
\|\theta_s(t'_2)\|_2 \le \frac{\|\theta_s(t'_1)\|_2}{\left(\frac{\nu H_1(\nu)^{\frac{p}{2}}}{2^{\frac{3p}{2}+1}} h^{-1} \left(\frac{H_1(\nu)^{-\frac{\alpha+\beta}{2}}}{2^{1-\frac{\alpha+\beta}{2}}} \right)^{\frac{p-2}{2}} (p-2)(t'_2-t'_1)\|\theta_s(t'_1) \|_2^{p-2}+1 \right)^{\frac{1}{p-2}}}.
\end{equation} 
Applying Lemma \ref{20210121lem021} with \eqref{20210121eq21} being replaced by \eqref{20210121eq03}, and \eqref{20210121eq22} being replaced by \eqref{20210121eq04}, respectively, we have
\begin{equation} \label{20210121eq05}
\|\theta_s(t_2')\|_2 \le \frac{\|\theta_{s, 0}\|_2}{\left( \frac{\nu H_1(\nu)^{\frac{p}{2}} {\mathcal H}_{1, \nu, h}}{2^{\frac{p}{2}}} \cdot (p-2)(t_2'-s) \|\theta_{s, 0}\|_2^{p-2}+1 \right)^{\frac{1}{p-2}}}.
\end{equation} 
It is then clear that \eqref{20210121eq02} follows by applying the above procedure $k$ times. 

\medskip

Finally, note that from \eqref{20210121eq02}, we can immediately concludes that
\begin{equation} \label{20210121eq11}
\kappa_d \le \frac{2^{\frac{p}{2}}}{\nu H_1(\nu)^{\frac{p}{2}}{\mathcal H}_{1, \nu, h}}+(t_0-s).
\end{equation} 
By the definition of $H_1(\nu)$ and the choice of $\lambda_N$ and $t_0$, there exists some constant $C'>0$, which only depends on $h$, $\|\nabla u\|_\infty$, $p$, $\|\theta_{0, 0}\|_2$, the strongly mixing condition and any dimension constants, such that 
\begin{equation} \label{20210121eq12}
t_0-s \le \frac{C'}{\nu H_1(\nu)^{\frac{p}{2}}} \le \frac{C'}{\nu H_1(\nu)^{\frac{p}{2}}{\mathcal H}_{1, \nu, h}},
\end{equation} 
where in the second estimate above, we have used the fact that ${\mathcal H}_{1, \nu, h} \le 1$. The desired estimate \eqref{20210121eq13} then clearly follows from \eqref{20210121eq11} and \eqref{20210121eq12}. 
\end{proof}

To prove Corollary \ref{strongcor}, it suffices to compute the function $H_1$ explicitly for the specific rate functions of interest. 

\begin{proof} [Proof of Corollary \ref{strongcor}]
We first note that when $\nu \ll 1$, ${\mathcal H}_{1, \nu, h} \equiv 1$. When the mixing rate function $h$ is given by the power law \eqref{powerlaw}, we compute 
\begin{equation} \label{20210224eq10}
H_1(\nu)=C_0 \abs{\ln \nu}^{\frac{2q}{\alpha+\beta}},
\end{equation} 
where $C_0=C_0(\alpha, \beta, c, p, d, \|\nabla u\|_{\infty}, \|\theta_{0, 0}\|_2)$. The desired estimate \eqref{corres1} then follows by substituting \eqref{20210224eq10} into \eqref{20210121eq13}. 

When the mixing rate function $h$ is given by the exponential law \eqref{explaw}, we have
$$
H_1(\nu)=\nu^{-\frac{2c_2}{pC_2+4\|\nabla u \|_{\infty}(\alpha+\beta)}}.
$$
This together with \eqref{20210121eq13} yields the desired estimate \eqref{corres2}. 
 \end{proof}

\medskip

\subsection{The Weakly Mixing Case.} We now turn to the proof of Theorem \ref{mainthm02}. The proof is similar to the proof of Theorem \ref{mainthm01}. The main difference is that the analog of Lemma \ref{20210121lem02} for the weakly mixing case is weaker. More precisely,we have the following result.

\begin{lemma} \label{20210224lem01}
Choose $\lambda_N$ to be the largest eigenvalue satisfying $\lambda \le H_2(\nu)$ where $H_2(\nu)$ is defined as in \eqref{20210121eq14new}. If
\begin{equation} \label{20210224eq11}
\|\grad \theta_{s, 0} \|_p<\lambda_N^{\frac{1}{2}} \|\theta_{s, 0}\|_2
\end{equation} 
then we have 
\begin{equation} \label{20210224eq12}
\|\theta_s(t)\|_2 \le \frac{\|\theta_{s, 0}\|_2}{\left(\frac{\nu H_2(\nu)^{\frac{p}{2}}}{2^{\frac{3p}{2}+1}} h^{-1} \left(\frac{H_2(\nu)^{-\frac{d+2\alpha+2\beta}{4}}}{2^{1-\frac{d+2\alpha+2\beta}{4}}} \right)^{\frac{p-2}{2}} (p-2)(t_0-s)\|\theta_{s, 0}\|_2^{p-2}+1 \right)^{\frac{1}{p-2}}}
\end{equation} 
at a time $t_0$ given by
$$
t_0:=s+2h^{-1} \left(\frac{ \lambda_N^{-\frac{d+2\alpha+2\beta}{4}}}{2 \sqrt{\mathfrak{c}}} \right),
$$
where $\mathfrak{c}$ is defined in \eqref{20210224eq100}. 
\end{lemma}

\begin{proof} [Proof of Theorem \ref{mainthm02}]
Given Lemma \ref{20210224lem01}, the proof of Theorem \ref{mainthm02} is the same as the proof of Theorem \ref{mainthm01}. 
\end{proof}

Moreover, as in the proof of Corollary \ref{strongcor}, the proof of Corollary \ref{weakcor} only involves computing $H_2$ explicitly when $h$ is assumed to be power law or exponential, and hence we would like to leave the detail to the interested reader. Finally, we proof Lemma \ref{20210224lem01}.

\begin{proof} [Proof of Lemma \ref{20210224lem01}]
The only difference between the proof of Lemma \ref{20210224lem01} and Lemma \ref{20210121lem02} is that when we estimate the term
$$
\int_{\frac{t_0+s}{2}}^{t_0} \|P_N \phi_s(r)\|_2^2dr, 
$$
instead of using the strongly mixing assumption, we need to bound it via the weakly mixing assumption \eqref{weakmix}. More precisely, we have
\begin{eqnarray} \label{20210224eq13}
\int_{\frac{t_0+s}{2}}^{t_0} \|P_N \phi_s(r)\|_2^2dr%
&\le& \int_{\frac{t_0+s}{2}}^{t_0} \sum_{\ell=1}^N \abs{\langle \phi_s(r), e_\ell \rangle}^2 dr \nonumber \\
&\le& \sum_{\ell=1}^N \frac{t_0-s}{2} h \left(\frac{t_0-s}{2} \right)^2 \|\phi_{s}(0)\|_{\dot{H}^{\alpha}}^2 \lambda_\ell^\beta \nonumber \\
&\le& \frac{N(t_0-s)}{2} h\left(\frac{t_0-s}{2} \right)^2 \lambda_N^\beta \|\theta_{s, 0}\|_{\dot{H}^\alpha}^2 \nonumber \\
&\le& \frac{N (t_0-s)}{2} h\left(\frac{t_0-s}{2} \right)^2 \cdot \lambda_N^{\alpha+\beta} \|\theta_{s, 0}\|_2^2 \nonumber  \\
&\le& \frac{\mathfrak{c} (t_0-s)}{2} h\left(\frac{t_0-s}{2} \right)^2 \cdot \lambda_N^{\frac{d+2\alpha+2\beta}{2}} \|\theta_{s, 0}\|_2^2. 
\end{eqnarray}
Here in the last estimate above, we have used the Weyl's formula \eqref{Weylfor} and Remark \ref{20210224rem01}.

Therefore, as the proof of Lemma \ref{20210121lem02}, we substitute \eqref{20210224eq13} and \eqref{20210120eq02} into \eqref{20201226eq03}, we see that 
\begin{eqnarray} \label{20210224eq21}
&& \int_s^{t_0} \|\nabla \theta_s(r)\|_2^2 dr \ge \lambda_N(t_0-s) \|\theta_{s, 0}\|_2^2 \cdot \bigg(\frac{1}{4}-\frac{\mathfrak{c}\lambda_N^{\frac{d+2\alpha+2\beta}{2}}}{4}h\left(\frac{t_0-s}{2} \right)^2 \nonumber \\
&& \quad \quad \quad \quad  \quad \quad  \quad \quad  \quad \quad \quad \quad  -\frac{\lambda_N^{\frac{p}{2}} \cdot d^{\frac{p-2}{2}} D_p \nu \|\theta_{0, 0}\|_2^{p-2}}{2\|\nabla u\|_\infty^2(t_0-s)} \cdot  e^{2\|\nabla u\|_\infty (t_0-s)} \bigg).
\end{eqnarray}
Finally, by the choice of $\lambda_N$ and $t_0$, we have
$$
\frac{\mathfrak{c}\lambda_N^{\frac{d+2\alpha+2\beta}{2}}}{4}h\left(\frac{t_0-s}{2} \right)^2 \le \frac{1}{16}, \quad \textrm{and} \quad \frac{\lambda_N^{\frac{p}{2}} \cdot d^{\frac{p-2}{2}} D_p \nu \|\theta_{0, 0}\|_2^{p-2}}{2\|\nabla u\|_\infty^2(t_0-s)} \cdot  e^{2\|\nabla u\|_\infty (t_0-s)} \le \frac{1}{16}, 
$$
 hence \eqref{20210224eq23} is still true in this case. The rest of the proof is then the same as those in the proof of Lemma \ref{20210121lem02}.
\end{proof}

\bibliographystyle{plain}
\bibliography{refs,preprints}

\end{document}